\documentclass[UKenglish, letterpaper]{amsart} \usepackage{amscd}
\usepackage{amssymb}
\usepackage{amsthm}
\usepackage{amsxtra}
\usepackage{babel}
\usepackage[T1]{fontenc}
\usepackage[utf8]{inputenc}
\usepackage{hyperref}
\usepackage[hyperpageref]{backref}
\usepackage{varioref}
\usepackage[arrow,curve,matrix]{xy}

\usepackage{mathpazo}
\usepackage{mathrsfs}

\usepackage{graphicx}
\usepackage{color}

\makeatletter

\newtheorem*{rep@theorem}{\rep@title}

\newcommand{\newreptheorem}[2]{%

\newenvironment{rep#1}[1]{%

\def\rep@title{#2 \ref{##1}}%

\begin{rep@theorem}}%

{\end{rep@theorem}}}

\makeatother

\def\clap#1{\hbox to 0pt{\hss#1\hss}}

\DeclareFontFamily{OMS}{rsfs}{\skewchar\font'60}
\DeclareFontShape{OMS}{rsfs}{m}{n}{<-5>rsfs5 <5-7>rsfs7 <7->rsfs10 }{}
\DeclareSymbolFont{rsfs}{OMS}{rsfs}{m}{n}
\DeclareSymbolFontAlphabet{\scr}{rsfs}

\newcommand{\sF}{\scr{F}}

\newcommand{\sO}{\scr{O}}

\newcommand{\sT}{\scr{T}}

\newcommand{\wtilde}{\widetilde}

\newcounter{thisthm}

\newcommand{\iref}[1]{(\thesection.\the\value{thisthm}.\the\value{#1})}

\theoremstyle{plain}    
\newtheorem{thm}{Theorem}[section]

\numberwithin{equation}{thm}
\numberwithin{figure}{section}

\theoremstyle{plain}    

\newtheorem{lem}[thm]{Lemma}
\newtheorem{conjecture}[thm]{Conjecture}

\theoremstyle{plain}    
\newtheorem{prop}[thm]{Proposition}
\newtheorem{proclaim-special}[thm]{\specialthmname}

\theoremstyle{remark}

\newtheorem{rem}[thm]{Remark}

\newtheorem{claim}[thm]{Claim} 
\newtheorem*{claim*}{Claim}

\newtheoremstyle{bozont-remark}{3pt}{3pt}%
     {}
     {}
     {\it}
     {.}
     {.5em}
     {\thmname{#1}\thmnumber{ #2}: \thmnote{\sc #3}}
\theoremstyle{bozont-remark}

\def\factor#1.#2.{\left. \raise 2pt\hbox{$#1$} \right/\hskip -2pt\raise
  -2pt\hbox{$#2$}} 
\newlength{\swidth}
\newenvironment{enumerate-p}{
  \begin{enumerate}}
  {\setcounter{equation}{\value{enumi}}\end{enumerate}}

\definecolor{tomato}{RGB}{180,62,39}
\definecolor{forrest}{RGB}{81,133,49}
\definecolor{lighttomato}{RGB}{253,65,65}
\definecolor{lightforrest}{RGB}{145,237,87}
\definecolor{mygreen}{RGB}{40,104,69}
\definecolor{mygreen2}{RGB}{3,149,39}
\definecolor{darkolivegreen}{RGB}{102,118,75}
\definecolor{cranegreen}{RGB}{102,118,75}
\definecolor{mydarkblue}{RGB}{10,92,153}
\definecolor{myblue}{RGB}{57,222,186}
\definecolor{pinkish}{RGB}{213,83,222}
\definecolor{colD}{RGB}{213,83,222}
\definecolor{defb}{RGB}{213,83,222}
\definecolor{goldenrod}{RGB}{225,115,69}
\definecolor{mauve}{RGB}{224, 176, 255}
\definecolor{fuchsia}{RGB}{255, 0, 255}
\definecolor{lavender}{RGB}{230, 230, 250}
\definecolor{gold}{RGB}{255, 215, 0}
\definecolor{orange}{RGB}{255, 127, 0}
\definecolor{maroon}{RGB}{123, 17, 19}
\definecolor{brightmaroon}{RGB}{195, 33, 72}
\definecolor{richmaroon}{RGB}{176, 48, 96}
\definecolor{green}{RGB}{3,149,39}

\title{Birational positivity in dimension $4$}

\author{Behrouz Taji \\ (with an appendix by Fr\'ed\'eric Campana)}

\address{Behrouz Taji, The Department of Mathematics, McGill University, Montreal, Canada.}

\email{\href{mailto:behrouz.taji@matil.mcgill.ca}{behrouz.taji@mail.mcgill.ca}}

\date{\today}


\begin{document}

\begin{abstract}
In this paper we prove that for a nonsingular projective variety of dimension at most 4 and with non-negative Kodaira dimension, the Kodaira dimension of coherent subsheaves of $\Omega^p$ is bounded from above by the Kodaira dimension of the variety. This implies the finiteness of the fundamental group for such an $X$ provided that $X$ has vanishing Kodaira dimension and non-trivial holomorphic Euler characteristic. 

\end{abstract}

\maketitle

\section{introduction}

It is a classical result of Bogomolov known as Bogomolov-De Franchis-Castelnuovo inequality that for projective varieties, Kodaira dimension of rank one subsheaves of $\Omega^p$ is bounded form above by $p$. In \cite{Ca95} Campana has proved that assuming some standard conjectures for non-uniruled varieties, the Kodaira dimension of such subsheaves admits another upper bound, namely the Kodaira dimension of the variety (See Theorem~\ref{campana}).

Throughout this paper we will refer to the conjectures of the minimal model program. See \cite{KM98} for the basic definitions and background. Here we will only state the ones that we need. 

\begin{conjecture}(The minimal model conjecture for nonsingular varieties) Let $X$ be a nonsingular projective variety over Z. If $K_X$ is pseudo-effective/ $Z$, then $X/Z$ has a minimal model. Otherwise it has a Mori fiber space/ $Z$.

\end{conjecture}

Remember that by $K_X$ pseudo-effective/$Z$, we mean $K_X$ can be realized as limit of effective divisors in the relative Neron-Severi space $N^1(X/Z)$.

\begin{conjecture}(The abundance conjecture for minimal models with terminal singularities) Let $X/Z$ be a normal projective variety with terminal singularities and $\mathbb{Q}$-Cartier canonical divisor . If $K_X$ is nef/ $Z$ then it is semi-ample/ $Z$, i.e. it is pull back of a divisor that is ample$/Z$.
\end{conjecture}

These two conjectures put together is sometimes referred to as the good minimal model conjecture.

For nonsingular projective varieties, Campana has introduced in \cite{Ca95} a new and a more general notion of Kodaira dimension defined by 
$$
  \kappa^+ \bigl(X \bigr) := \max \left\{ \kappa \bigl( \det \sF \bigr)
    \, \bigl| \, \sF \text{ is a coherent subsheaf of $\Omega^p_{X}$,
      for some $p$} \right\}
  $$

and conjectured that $\kappa=\kappa ^+$ when $\kappa \geq 0$. He proves that the conjectured equality holds assuming the good minimal model conjecture:

\begin{thm}[\protect{Equality of $\kappa$ and $\kappa^+$ when $\kappa \geq 0$,  cf.~\cite[Prop. 3.10]{Ca95}}]\label{campana} Let X be a nonsingular projective variety in dimension $n$ with non-negative Kodaira dimension. If the good minimal model conjecture holds for nonsingular projective varieties  of dimension up to $n$ and with vanishing Kodaira dimension, then $\kappa(X)=\kappa^+(X)$.

\end{thm}

This in particular refines Bogomolov's inequality when the Kodaira dimension is relatively small, for example when $\kappa(X)=0$ and $L\subseteq \Omega^p_X$, then $\kappa(L)\leq 0$. 

Note that when $c_1=0$ then we have $\kappa=\kappa^+$ by Bochner's vanishing coupled with Yau's solution \cite{Yau77} to the Calabi's conjecture.

By Theorem 1.3, $\kappa(X)$ and $\kappa^+(X)$ coincide for non-uniruled threefolds as a consequence of the minimal model program or MMP for short (See for example \cite{Ko92}). We prove Campana's conjecture in dimension four and for varieties with positive Kodaira dimension in dimension five:

\begin{thm}\label{equality} Let $X$ be a nonsingular projective variety. 

(i). If dimension of $X$ is at most 4 and $\kappa(X) \geq 0$, then $\kappa=\kappa^+$.

(ii). If dimension of $X$ is 5 and $\kappa(X) \geq 1$, then $\kappa=\kappa^+$.

\end{thm}

\
Furthermore one can show that the cotangent bundle of such varieties is \emph{birationally stable} (See the appendix~\ref{stability}). 

\

Theorem 1.4 is a consequence of a much more general result that we obtain in this paper:

\begin{thm}\label{general} Let X be a nonsingular projective variety of dimension $n$. Assume that the good minimal model conjecture holds for terminal projective varieties with zero Kodaira dimension up to dimension $n-m$, where $m>0$. If $\kappa(X)\geqslant m-1$ then $\kappa=\kappa^+$.

\end{thm}

An important corollary of~\ref{equality} is the finiteness of the fundamental group of 4-dimensional varieties with vanishing Kodaira dimension and non-zero holomorphic Euler characteristic (See~\ref{finiteness} below) which follows from a remarkable result of Campana:

\begin{thm}[\protect{Finiteness of the fundamental groups,  cf.~\cite[Cor. 5.3]{Ca95}}]\label{observation} Let X be a nonsingular projective variety. If $\kappa^+(X)=0$ and $\chi(X,\sO_X) \neq 0$, then $\pi_1(X)$ is finite. 

\end{thm}

\begin{thm}\label{finiteness} Let $X$ be a nonsingular projective variety of dimension at most $4$. Assume $\kappa(X)=0$ and  $\chi(X,\sO_X) \neq 0$, then $\pi_1(X)$ is finite.

\end{thm}

\

\subsection{Acknowledgements} The author would like to thank his advisor S. Lu for his advice, guidance and constant support. A special thanks is owed to F. Campana for his suggestions and encouragements. The author also wishes to express his gratitude to the anonymous referee for the insightful comments.

\

\section{generic semi-positivity and pseuodo-effectivity}

Let $X$ be a non-uniruled nonsingular projective variety. It is a well known result of Miyaoka, cf.~\cite{Miy87,Miy85} that $\Omega_X$ is generically semi-positive. This means that the determinant line bundle of any torsion free quotient of $\Omega_X$ has non-negative degree on curves cut out by sufficiently ample divisors. Equivalently we can characterize this important positivity result by saying that $\Omega_X$ restricted to these general curves is nef unless $X$ is uniruled. This property is sometimes called \emph{generic nefness}. Since nefness is invariant under taking symmetric powers this result automatically generalizes to $\Omega^p_X$. Using the same characteristic $p$ arguments as Miyaoka and some deep results in differential geometry, Campana and Peternell have shown that in fact such a determinant line bundle is dual to the cone of moving curves, i.e. its restriction to these curves has non-negative degree. By \cite{BDPP} this is the same as saying that it is pseudo-effective.

	\begin{thm}[\protect{Pseudo-effectivity of quotients of $\Omega^p_X$, cf.~\cite[Thm. 1.7]{CP07}}]\label{pe} Let X be a non-uniruled nonsingular projective variety and let $\sF$ be an $\sO_X$- module torsion free quotient of $\Omega^p_X$. Then $\det\sF$ is a pseudo-effective line bundle.

\end{thm}


\

\section{The Refined Kodaira dimension}

In this section we will use more or less the same ideas as Cascini \cite{Cs06} to show that $\kappa$ and $\kappa^+$ coincide for nonsingular projective varieties of dimension four with non-negative Kodaira dimension and also for varieties of dimension five with positive Kodaira dimension. The following proposition is a result of Campana, cf.~\cite{Ca95}. We include a proof for completeness.

\begin{prop} Let $X$ be a nonsingular projective variety with $\kappa(X)=0$. If $X$ has a good minimal model then $\kappa=\kappa^+$.

\end{prop}

\begin{proof} Let $Y$ be a $\mathbb{Q}$-factorial normal variety with at worst terminal singularities serving as a good minimal model for $X$.  Note that $K_{Y}$ is numerically trivial. Let $\pi:\wtilde Y \to Y$ be a resolution. Since $\kappa(\wtilde Y)=0$, $\wtilde Y$ is not uniruled. Let $\sF\subseteq \Omega_{\wtilde Y}^{p}$ be a coherent subsheaf with maximum Kodaira dimension, i.e. $\kappa(\det\sF)=\kappa^+(\wtilde Y)$. 

Let $C$ be an irreducible curve on $Y$ cut out by sufficiently general hyperplanes and let $\wtilde C$ to be the corresponding curve in $\wtilde Y$ . Now using the standard isomorphism:  $\Omega_{\wtilde Y}^p|_{\wtilde C} \cong K_{\wtilde Y}|_{\wtilde C} \otimes \wedge^{n-p} \sT_{\wtilde Y}|_{\wtilde C}$ , we get $\sF^*|_{\wtilde C}$ as a quotient of $K_{\wtilde Y}^*|_{\wtilde C} \otimes \Omega_{\wtilde Y}^{n-p}|_{\wtilde C}$. But $K_{\wtilde Y}^*$ is numerically trivial on $\wtilde C$ and $\Omega_{\wtilde Y}^{n-p}|_{\wtilde C}$ is nef by Miyaoka, so $\sF^*|_{\wtilde C}$ must also be nef and we have $$\deg(\det\sF|_{\wtilde C})\leq 0.$$  But this inequality holds for a covering family of curves and thus $\kappa(\sF) \leq 0$. 

\end{proof}

As was mentioned in the introduction (Theorem~\ref{campana}), assuming the good Minimal Model conjecture for varieties up to dimension $n$ and with zero Kodaira dimension, we have $\kappa=\kappa^+$ in the case of $n$-dimensional varieties of positive Kodaira dimension as well. See ~\cite[Prop. 3.10]{Ca95} for a proof. The main result of this paper is concerned with replacing this assumption with the abundance conjecture in lower dimensions.

\begin{rem} Following the recent developments in the minimal model program, we now know that we have a good minimal model when numerical Kodaira dimension is zero. The proposition 3.1 shows that $\kappa^+(X)$ also vanishes in this case. By \cite{Ca95} this implies in particular that nonsingular varieties with vanishing numerical dimension have finite fundamental groups as long as they have non-trivial holomorphic Euler characteristic (See~\ref{observation}).

\end{rem}

We will need the following lemmas in the course of the proof of our main result. 

\begin{lem}\label{lemma1} Let  $f:X \to Z$ be a surjective morphism with connected fibers between normal projective varieties $X$ and $Z$. Let $D$ be an effective $\mathbb{Q}$-Cartier divisor in $X$ that is numerically trivial on the general fiber of $f$. If D is $f$-nef, then there exist  biratioanl morphisms $\pi:\wtilde Z \to Z$, $\mu: \wtilde X \to X$, a $\mathbb{Q}$-Cartier divisor $G$ in $\wtilde Z$, and an equidimensional morphism $\wtilde f:\wtilde X \to \wtilde Z$ such that $\mu^*(D)=\wtilde f^*(G)$.

\end{lem}

\begin{proof} The fact that we can modify the base of our fibration to get a morphism whose fibers are of constant dimension is guaranteed by \cite{Ray72}. This is called \emph{flattening} of $f$. Let $\wtilde X$ be a normal birational model of $X$ and $\wtilde Z$ a smooth birational model for $Z$ such that $\wtilde f: \wtilde X \to \wtilde Z$ is flat. 

If general fibers are curves, by assumption the degree of $\mu^*(D)$ on general fibers of $\wtilde f$ is zero. On the other hand $\mu^*(D)$ is effective and relatively nef, so it must be trivial on all fibers. This implies the existence of the required $\mathbb{Q}$-Cartier divisor $G$ in $\wtilde Z$.

In the case of higher dimensional fibers, $\mu^*(D)$ must still be numerically trivial on all fibers of $\wtilde f$. To see this, let $C$ be an irreducible curve contained in a $d$-dimensional non-general fiber $\wtilde F_0$ of $\wtilde f$. Then for a sufficiently general members $D_i$ of the linear system of an ample divisor $H$, we have
$$D_1 ...  D_{d-1}. \wtilde F_0=mC+C',$$
where $C'$ is an effective curve and $m$ accounts for the multiplicity of the irreducible component of $F_0$ containing $C$. Now since $\mu^*(D)$ is numerically trivial on the general fiber of $\wtilde f$, we have   $$\mu^*D.(mC+C')=0.$$ But $\mu^*(D)$ is $\wtilde f$-nef, so that $\mu^*(D).C=0$.

We know that $\mu^*(D)$ is effective, so $\mu^*(D)$ must be trivial on all fibers. Again this ensures the existence of a $\mathbb{Q}$-Cartier divisor $G$ in $\wtilde Z$ such that $\mu^*(D)=\wtilde f^*(G)$.

\end{proof}

In the course of the proof of lemma 3.3 we repeatedly used the standard fact that given a surjective morphism $f:X \to Z$ between normal varieties $X$ and $Z$, where $Z$ is $\mathbb{Q}$-factorial, and an effective $\mathbb{Q}$-Cartier divisor $D$ that is trivial on all fibers, we can always find a $\mathbb{Q}$-Carier divisor $G$ in $Z$ such that $D=f^*(G)$. One can verify this by reducing it to the case where $X$ is a surface and $Z$ is a curve. Here the negative semi-definiteness of the intersection matrix of the irreducible components of singular fibers establishes the claim.

\

For application a natural setting for lemma~\ref{lemma1} is the relative minimal model program. The following is a reformulation of this lemma in this context. 

\

\begin{lem}\label{lemma2} Let $f:X \to Z$ be a surjective morphism with connected fibers between nonsingular projective varieties $X$ and $Z$ with dimension $n$ and $m$ respectively. Assume $\kappa(X) \geq 0$ and that $X/Z$ has a minimal model model $Y/Z$. Denote the morphism between $Y$ and $Z$ by $\psi$. Also assume that the abundance conjecture for varieties of vanishing Kodaira dimension holds in dimension $n-m$. If the Kodaira dimension of the general fiber of $f$ is zero, then there exist birational morphisms $\pi: \wtilde Z \to Z$, $\mu: \wtilde Y \to Y$, a $\mathbb{Q}$-Cartier divisor G in $\wtilde Z$, and an equidimensional morphism $\wtilde \psi: \wtilde Y \to \wtilde Z$ such that $\mu^*(K_Y)=\wtilde \psi^*(G)$. 

$$
  \xymatrix{
    X \ar@{.>}[r]  \ar[dr]  & Y  \ar[d]^{\psi}  & \wtilde Y  \ar[l]_{\mu} \ar[d]^{\wtilde {\psi}}  \\
      & Z   & \wtilde Z \ar[l]^{\pi}
  }
  $$

\end{lem}

\begin{proof} Since $K_Y$ is $\psi$-nef and that the dimension of the general fibers is $n-m$, we find that  the canonical of the general fiber is torsion by the abundance assumption. Now apply lemma 3.3 to $\psi :Y \to Z$ and take $K_Y$ to be $D$.

\end{proof}

We now turn to another crucial ingredient that we shall use in the proof of~\ref{proposition}. 

\begin{lem} Let $f:X \to Z$ be a surjective morphism with connected fibers between normal projective varieties $X$ and $Z$ of dimension $n$ and $k$ respectively. Let $D$ be a $\mathbb{Q}$-Cartier divisor in $Z$. If $f^*D$ is dual to the cone of moving curves then so is $D$.

\end{lem}

\begin{proof} First assume that $f$ is birational. Let $C$ be a moving curve in $Z$ and let $\mu: \wtilde Z \to Z$, be a birational morphism such that $\mu_*(\wtilde C)=C$, where $\wtilde C$ is a complete intersection curve cut out by hyperplanes. Let $\pi: \wtilde X \to X$ be a suitable modification such that $\wtilde f: \wtilde X \to \wtilde Z$ is a morphism and we have the following commutative diagram.

 $$
  \xymatrix{
    \wtilde X \ar[r]^{\wtilde f }  \ar[d]_{\pi} & \wtilde Z  \ar[d]^{\mu} \\
      X \ar[r]^f       & Z
  }
  $$

Now let $\wtilde C=H_1 ...  H_{k-1}$, where $H_1$,..., $H_{k-1}$ are ample divisors in $\wtilde Z$. We have

\begin{align*}
\mu^*D.\wtilde C &= \mu^*D.H_1 ... H_{k-1}\\
        &=\wtilde f^*(\mu^*D). \wtilde f^*H_1 ...  \wtilde f ^*H_{k-1}\\
        &=\pi^*(f^*D). \wtilde f^*H_1 ...  \wtilde f ^*H_{k-1} && \text{by commutativity of the diagram.}\\
      \end{align*}
 
 Clearly $\pi^*(f^*D)$ is pseudo-effective. Now since nef divisors are numerically realized as limit of ample ones we have 
$$\pi^*(f^*D). \wtilde f^*H_1 ...  \wtilde f ^*H_{k-1} \geq 0,$$
which implies $\mu^*(D).\wtilde C \geq 0$. So that $D.C \geq 0$ as required.

Now assume that $f$ is not birational and let $C=H_1...H_{k-1}$ be an irreducible curve cut out by ample divisors in $Z$. In particular $C$ is of constant dimension along the image of fibers. After cutting down by general hyperplanes $H'_1$,..., $H'_{n-k}$, we can find an irreducible curve $$C'=H'_1...H'_{n-k}.f^*(H_1)...f^*(H_{k-1}) $$ that maps surjectively onto $C$. Thus we have $(\deg f|_{C'})D.C=f^*D. C' \geq 0$. 

For a moving curve that is not given by intersections of hyperplanes, we repeat the same argument as above after going to a suitable modification.
\end{proof}

\

\begin{rem} We know by \cite{BDPP} that for nonsingular projective varieties, pseudo-effective divisors are dual to the cone of moving curves. Using the lemma above, we can easily extend this to normal varieties by going to a resolution. This fact is of course already well known. For convenience  we rephrase the lemma 3.5 as follows:

\

\emph{Lemma 3.5'}. Let $f:X \to Z$ be a surjective morphism with connected fibers between normal projective varieties. Let $D$ be a $\mathbb{Q}$-Cartier divisor in $Z$. If $f^*D$ is \emph{pseudo-effective} then so is $D$. 

\end{rem}

\

We shall prove Theorem~\ref{general} as a consequence of the following proposition:

\begin{prop}\label{proposition} Let X be a nonsingular projective variety of dimension $n$ with $\kappa(X)\geqslant0$. 
Assume that the good minimal model conjecture holds for terminal projective varieties with zero Kodaira dimension up to dimension $n-m$, where $m>0$. Let $\sF\subseteq \Omega^p_X$ be a coherent subsheaf and define the line bundle $L=\det \sF$. If $\kappa(K_X+L)\geqslant m$, then $\kappa(L)\leqslant \kappa(X)$. 

\end{prop}

\begin{proof} First a few observations. The  isomorphism $K_X^*\otimes \Omega^p_X \cong \wedge^{n-p} \sT_X$ implies that $K^*_X\otimes \sF$ is a subsheaf of $\wedge^{n-p} \sT_X$.  But X is not uniruled and so by~\ref{pe} $rK_X-L$  is pseudo-effective as a Cartier divisor, where $r$ is the rank of $\sF$.

We can of course assume that X is not general type. Now if we assume that $K_X+L$ is big then by using the  equality $(r+1)K_X=(rK_X-L) + (K_X+L)$ and pseudo-effectivity of $rK_X-L$, we conclude that $K_X$ must be big as well. So we may also assume that $K_X+L$ is not big and that $\kappa(L)>0$. 

Without loss of generality we can also assume that the rational map $X\dashrightarrow Z$ corresponding to $K_X+L$ is a morphism, since we can always go to a suitable modification, pull back $L$ and prove the theorem at this level. Denote this map by $i_{K_X+L}$ and note that by definition we have $\kappa((K_X+L)|_F)=0$ , where $F$ is the general fiber of $i_{K_X+L}$. Finally, we observe that $\kappa(F) \leq \kappa((K_X+L)|_F)=0$ and as we are assuming that $X$ has non-negative Kodaira dimension, we have $\kappa(F)=0$.

\begin{claim}\label{claim} Without loss of generality, we can assume $L$ is the pull back of a $\mathbb{Q}$-Cartier divisor $L_1$ in $Z$. 
\end{claim}

Assuming this claim for the moment, our aim is now to show that after a modification $\pi:\wtilde Z \to Z$, we can find a big divisor in $\wtilde Z$ whose Kodaira dimension matches that of $X$. This will imply that $\kappa(L)=\kappa(L_1) \leq \kappa(X)$, as required.

To this end take $Y$ to be a relative minimal model for $X$ over $Z$ and denote the birational map between $X$ and $Y$ by $\phi$ and the induced morphism $Y \to Z$ by $\psi$ (See the diagram below). Observe that we can assume that $\phi$ is a morphism without losing generality. Denote $\psi^*(L_1)$ by $L_Y$. Fix $K_Y$ to be the cycle theoretic push forward of $K_X$. 

Now by lemma~\ref{lemma2} and the abundance assumption after modifying the base by $\pi: \wtilde Z \to Z$, we can find a morphism $\wtilde {\psi}:\wtilde Y \to \wtilde Z$ such that the dimension of the fibers of this new fibration are all the same and $\mu^*(K_Y)=\wtilde \psi^*(G)$ for some $\mathbb{Q}$-Cartier divisor $G$ in $\wtilde Z$.

$$
  \xymatrix{
    X \ar[r]^{\phi}  \ar[dr]_{i_{K_X+L}} & Y  \ar[d]^{\psi}  & \wtilde Y  \ar[l]_{\mu} \ar[d]^{\wtilde {\psi}}  \\
      & Z   & \wtilde Z \ar[l]^{\pi}
  }
  $$
Noting that $Y$ is at worst terminal, i.e. $K_X+L=\phi^*(K_Y+L_Y)+E$ for an effective exceptional divisor $E$, we have $\kappa(K_X+L)=\kappa(K_Y+L_Y)$. We also observe that $rK_Y -L_Y$ must be pseudo-effective. 

Define $ \wtilde L_1:=\pi^*(L_1)$, so that $\mu^*(L_Y)=\wtilde \psi^*(\wtilde L_1)$ and $\wtilde \psi^*(G+\wtilde L_1)=\mu^*(K_Y+L_Y)$. This implies that $G+\wtilde L_1$ is big in $\wtilde Z$. We also know that $\mu^*(rK_Y-L_Y)$ is pseudo-effective and $\mu^*(rK_Y-L_Y)=\wtilde \psi^*(rG-\wtilde L_1)$. Thus by lemma 3.4, $rG-\wtilde L_1$ is pseudo-effective too. Additionally we have 
$$(r+1)G=(rG-\wtilde L_1)+(G+\wtilde L_1),$$

where the right hand side is a sum of pseudo-effective and big divisors. This implies that G is big and we have $$\kappa(L)=\kappa(\wtilde L_1) \leq \kappa(G)=\kappa(\mu^*(K_Y))=\kappa(K_X).$$

\end{proof}

Now it remains to prove~\ref{claim}. 

\begin{proof}[Proof of~\ref{claim}] Let $X \dashrightarrow Z'$ be the map given by the global sections of large enough multiple of L, and let $i_L:X' \to Z'$ be the Iitaka fibration corresponding to L, where $\mu:X' \to X$ is a suitable modification of $X$. As $\kappa(K_X) \geq 0$, we have $\kappa(L) \leq \kappa(K_X+L)$, where the right hand side of this inequality is zero on the general fiber of $i_{K_X+L}$. On the other hand since we have assumed $\kappa(L)$ to be positive, we find that $\kappa(L|_F)=0$. Hence $i_{K_X+L}$ factors through $i_L$ via a rational map $g$ and we have the following commutative diagram:

 $$
 \xymatrix{
 X' \ar[d]_{\mu}  \ar[r]^{i_L} & Z' \\
 X \ar[r]_{i_{K_X+L}}  \ar@{.>}[ur] & Z \ar@{.>}[u]_{g}
 }
 $$

Now by considering suitable modifications of $X$, $Z$ and $X'$, we can assume that $g$ is a morphism. Define the line bundle $L':= \mu^*(L)-A=i_L^*(H)$, where A is an effective divisor and H is an ample $\mathbb{Q}$-Cartier divisor in $Z'$. Let $L"$ be the pull back of $H$ in $X$ via $g$ and $i_{K_X+L}$, so that $\mu^*(L")=L'$ and that $\mu^*(L")+A=\mu^*(L)$. We claim that we don't lose generality if we replace $L$ by $L"$. To see this we need to check the following two properties: (i) $rK_X-L"$ is pseudo-effective and (ii) $\kappa(L")=\kappa(L)$.

To see that (i) holds, note that we have $\mu^*(rK_X-L")=\mu^*(rK_X)-(\mu^*(L)-A)=\mu^*(rK_X-L)+A$. Now since $rK_X-L$ is pseudo-effective and $A$ is effective, $rK_X-L"$ must also be pseudo-effective. 

 For (ii) it suffices to show $\kappa(L)=\kappa(L')$ which is a consequence of the following inequality:

$$\kappa(L)=\kappa(\mu^*L) \leq dim Z'=\kappa(L').$$

This finishes off the proof of Claim~\ref{claim} after a possible base change corresponding to $K_X+L"$.

\end{proof}

Now our main result immediately follows:

\begin{proof}[Proof of Theorem~\ref{general}] Let $\sF \subseteq \Omega_X^p$ be a coherent subsheaf with maximum Kodaira dimension, i.e. $\kappa(L)=\kappa^+(X)$, where $L=\det(\sF)$. Assume that $\kappa(L)>\kappa(X)$. Then $\kappa(L) \geq m$ and in particular we have $\kappa(K_X+L) \geq m$. Now the proposition above implies that $\kappa(L) \leq \kappa(X)$, which is a contradiction.

\end{proof}

\

As we discussed in the introduction, this greatly improves the Bogomolov's inequality for projective varieties of dimension at most five and with relatively small Kodaira dimension. 

\

\begin{rem}[Birational stability in dimension $4$]\label{remark}Theorem~\ref{equality} can be further strengthened by replacing $\kappa^+$ by a stronger birational invariant $\omega(X)$ (See~\ref{stability} for the definition) which measures the maximal positivity of coherent rank one subsheaves of $\Omega^1_X$$^{\otimes m}$, for any $m>0$, i.e. $\kappa(X)$ and $\omega(X)$ coincide for fourfolds with non-negative Kodaira dimension. The proof is identical to that of Theorem~\ref{equality} by observing that pseudo-effectivity of $rK_X-L$ in~\ref{proposition} can be be replaced by that of $mK_X-L$, where m denotes the tensorial power of cotangent bundle containing the line bundle $L$ .

\end{rem}
\

\begin{rem}We would like to point out that when $\kappa(X) \geq  \dim X - 3$, we have $\kappa=\kappa^+$ by \cite[Prop. 10.9]{Ca95} , where $3$ in this inequality comes from the abundance result for varieties of dimension at most $3$. So the real improvement  provided by~\ref{equality} is when $\kappa=0$ in dimension $4$ and $\kappa=1$ in dimension $5$. 
\end{rem}

\

\section{Appendix: Birational Stability of the Cotangent Bundle \\ (by Fr\'ed\'eric Campana)}

We present here a new birational invariant $\omega(X)$ similar to the Kodaira dimension $\kappa(X)$, at least equal to $\kappa(X)$ and to our previous $\kappa^+(X)$ and $\kappa_+(X)$, and conjecturally equal to all these when $X$ is not uniruled. The preceding arguments of Behrouz Taji directly apply to show this conjecture in dimension $4$ as well, and also under the situations considered in his theorem \ref{general} above. This invariant can be introduced in the orbifold compact K\" ahler case as well, with similar expected properties.

 A major aim of algebraic geometry consists in deriving the qualitative geometry of a complex connected projective manifold $X$ from the positivity/negativity properties of its canonical bundle $K_X$ (e.g: it has been shown that $X$ is rationally connected, hence simply connected when its canonical bundle is anti-ample). An intermediate step consists in relating the positivity/negativity of its cotangent bundle $\Omega^1_X$ from the one of $K_X$.

The positivity of $K_X$ is suitably measured by the canonical (or Kodaira) dimension of its canonical algebra $K(X):=\oplus_{m\geq 0}H^0(X,K_X^{\otimes m})$, defined as: $\kappa(X):=max_{\{m>0\}}(dim\Phi_m(X))\in\{-\infty, 0,1,\dots,n\}$, where $\Phi_m:X\dasharrow \Bbb P(H^0(X,K_X^{\otimes m})^*)$ is the rational map defined by the linear system $H^0(X,K_X^{\otimes m})$ if this is nonzero, and is $-\infty$ otherwise.

In a similar way, we define: $\Omega(X):=\oplus_{m\geq 0}H^0(X,\Omega^1_X$$^{\otimes m})$, to be the the cotangent algebra of $X$, and its dimension to be:
 $\omega(X):=max_{\{m>0, L\}}(dim \Phi_{ L}(X))\in \{-\infty, 0,1,\dots,n\}$, where $L\subset \Omega^1_X$$^{\otimes m}$ ranges over all of its coherent rank one subsheaves, with $m>0$ arbitrary. Here, $\Phi_LX\dasharrow \Bbb P(H^0(X,L)^*)$ is the rational map associated with the linear system defined by the sections of $L$.
  
  \
  
  {\bf Basic properties:}\label{stability} 1. $\omega(X)\geq \kappa(X)$ is a birational invariant of $X$. We say that $\Omega^1_X$ is birationally stable if $\omega(X)=\kappa(X)$, which means that $\kappa(X,L)\leq \kappa(X)$, for any $L\subset \Omega^1_X$$^{\otimes m}$, coherent of rank $1$, if $m>0$ is arbitrary. We shall see below that this happens, conjecturally, if and only if $X$ is not uniruled (or if and only if $\kappa(X)\geq 0$).
  
  2. Recall that we defined in \cite{Ca95} and \cite{Ca04} two other invariants: 
  $\kappa^+(X):=max_{\{F\subset \Omega^p_X,p>0\}}(\kappa(det(F))$, and:
  $\kappa_+(X):=max_{\{L\subset \Omega^p_X,p>0\}}(\kappa(L))$, where $F$ and $L$ are arbitrary coherent subsheaves, $L$ being of rank one.
   
    We thus have: $\omega(X)\geq \kappa^+(X)\geq \kappa_+(X)\geq \kappa(X)$ in general.
  
  3. Moreover, $\Omega(X)$ is (in contrast to $K(X)$) functorial in $X$, that is: any dominant rational map $f: X\dasharrow Y$ induces naturally an injective algebra morphism $f^*:\Omega(Y)\to \Omega(X)$, and so  $\omega(X)$ is increasing: $\omega(X)\geq \omega(Y)$ in this situation. The invariant $\omega$ thus provides an obstruction to the existence of such an $f$.
  
  4. Let $f:X\to Y$ be a surjective morphism between projective manifolds $X,Y$. Let $X_y$ be its general fibre. If $\omega(X_y)=-\infty$ (resp. if $\omega(X_y)=0$), then $\omega(X)=\omega(Y)$ (resp. $\omega(X)\leq dim(Y)$). We omit the easy proof. 
  
  This inequality is, in general, optimal, as seen by considering the Moishezon-Iitaka fibration $f:X\to Y$ of a manifold $X$ with $\kappa(X)\geq 0$, such that $c_1(X_y)=0$ (the examples below indeed show that $\omega(X_y)=\kappa(X_y)=0$, then). 
  
  If $r_X:X\to R(X)$ is the rational quotient of $X$ (defined in \cite{Ca92}, and called `MRC fibration in \cite{KMM92}), the first statement implies that $\omega(X)=\omega(R(X))$. Recall that $r$ is the only fibration on $X$ having rationally connected fibres and non-uniruled base $R(X)$ (by \cite{GHS01}). 
  
  \
  
  {\bf Examples:} 1. If $X$ is rationally connected, $\omega(X)=-\infty$. Conjecturally, the opposite implication holds as well, and follows from the Abundance conjecture. 
  
  2. If $c_1(X)=0$, then $\omega(X)=0$, as seen from the existence of a Ricci-flat K\"ahler metric and Bochner principle. Alternatively, the general semi-positivity theorem of Myiaoka shows that $\omega (X)=0$ if $X$ has a birational (normal) model $X'$ such that $K_{X'}$ is numerically trivial over its non-singular locus (i.e.: such that its degree is zero on each projective curve not meeting its singular locus). The existence of such a model is implied by the Abundance conjecture, too. The Abundance conjecture (and Miyaoka's theorem) thus imply that $\omega(X)=0$ if $\kappa(X)=0$.
  
  \
  
  {\bf Conjecture:} For any $X$, we have: $\omega(X)=\kappa(R(X))$ (by convention: $\kappa(pt)=-\infty$).

 \
  
 {\bf Remark:} This conjecture follows from the Abundance conjecture. Indeed: we need only check that $\omega(R(X))=\kappa(R(X))$, since $\omega(X)=\omega(R(X))$, by the property 4 above. Since $R(X):=Y$ is not uniruled, we have: $\kappa(Y)\geq 0$, by Abundance. Let $f:Y\to Z$ be the Moishezon-Iitaka fibration of $Y$. Since its general fibres have $\kappa=0$, they have $\omega=0$ as well by Abundance (example 2 above). Thus $\omega(Y)\leq dim(Z)=\kappa(Y)\leq \omega(Y)$. Thus $dim(Z)=\kappa(R(X))=\omega(Y)=\omega(R(X))$ $\square$
 
 \
 
 {\bf Remark:} The definition of $\Omega(X)$ and $\omega(X)$ with entirely similar properties can be extended to the case when $X$ is compact K\" ahler, and more importantly, when $X$ (possibly compact K\" ahler) is equipped with an orbifold divisor $\Delta:=\sum_ja_j.D_j$, where the $D_j's$ are irreducible pairwise distinct divisors on $X$ whose union is of normal crossings, and the $a_j's$ are in $\Bbb Q\cap[0,1]$. See \cite{Ca04} and \cite{Ca11} for the relevant definitions. The details will be written elsewhere.

\providecommand{\bysame}{\leavevmode\hbox to3em{\hrulefill}\thinspace}
\providecommand{\MR}{\relax\ifhmode\unskip\space\fi MR }
\providecommand{\MRhref}[2]{%
  \href{http://www.ams.org/mathscinet-getitem?mr=#1}{#2}
}
\providecommand{\href}[2]{#2}


\begin{thebibliography}{BDPP04}



\bibitem[BDPP04]{BDPP}
S{\'e}bastien Boucksom, Jean-Pierre Demailly, Mihai Paun, and Thomas Peternell, \emph{The pseudo-effective cone of a compact {K}\"ahler manifold and varieties of negative Kodaira dimension}, J. Algebraic. Geom. \textbf{22} (2013), 201-248. doi:10.1090/S1056-3911-2012-00574-8.	


\bibitem[Ca92]{Ca92} 
Fr{\'e}d{\'e}ric Campana, \emph{Connexit\'e rationnelle des vari\'et\'es de Fano.} {\sl Ann. Sc. ENS 25}, 539-545 (1992).



\bibitem[Cam95]{Ca95}
Fr{\'e}d{\'e}ric Campana, \emph{Fundamental group and positivity of cotangent bundles of compact
  {K}\"ahler manifolds}, J. Algebraic Geom. \textbf{4} (1995), no.~3, 487--502.
  \MR{1325789 (96f:32054)}

\bibitem[Ca04]{Ca04}
Fr{\'e}d{\'e}ric Campana, \emph{Orbifolds, special varieties and classification theory.}
{\sl Ann.\ Inst.\ Fourier (Grenoble)}, 54(3):499--630, 2004.


\bibitem[Ca11]{Ca11} 
Fr{\'e}d{\'e}ric Campana, \emph{Orbifoldes g\'eom\'etriques sp\'eciales et classification bim\'eromor-phe des vari\'et\'es K\" ahl\'eriennes compactes}, {\sl J. Inst. Math. Jussieu 10}, 809-934, 2011.


\bibitem[CPT07]{CP07}
Fr{\'e}d{\'e}ric Campana, Thomas Peternell, and Matei Toma, \emph{Geometric stability of the cotangent bundle and the universal cover of a projective manifold}, 	Bull, Soc. Math. \textbf{139} (2011) 41-74.

\bibitem[Cas06]{Cs06}
Paolo Cascini, \emph{Subsheaves of the cotangent bundle}, CEJM 4(2) 2006, 209-224. doi: 10.2478/S11533-006-0003-Z.


\bibitem[GHS01]{GHS01} 
Tom Graber, Joe Harris, Jason Starr, \emph{Families or rationally connected varieties.}{\sl J. Amer. Math. Soc. 16}, 57-67 (2003).


\bibitem[Iit82]{Ii82}
Shigeru Iitaka, \emph{Algebraic Geometry}, Graduate Texts in Math., Vol. 76, Springer, 1982.


\bibitem[Ko92]{Ko92}
J{\'a}nos {\hbox{K}}oll{\'a}r, \emph{Flips and Abundance for Algebraic Threefolds}, Ast\'erisque, 211, 1992.

\bibitem[KMM92]{KMM92} 
J\'anos Koll\'ar, Yoichi Miyaoka, S. Mori, \emph{Rationally connected Manifolds.} {\sl J. Alg. Geom. 1}, 429-448, (1992).

\bibitem[{\hbox{K}}{\hbox{M}}98]{KM98}
J{\'a}nos {\hbox{K}}oll{\'a}r and Shigefumi {\hbox{M}}ori, \emph{Birational
  geometry of algebraic varieties}, Cambridge Tracts in Mathematics, vol. 134,
  Cambridge University Press, Cambridge, 1998.


\bibitem[Miy87a]{Miy87}
Yoichi Miyaoka, \emph{The {C}hern classes and {K}odaira dimension of a minimal
  variety}, Algebraic geometry, Sendai, 1985, Adv. Stud. Pure Math., vol.~10,
  North-Holland, Amsterdam, 1987, pp.~449--476. \MR{89k:14022}

\bibitem[Miy87b]{Miy85}
\bysame, \emph{Deformations of a morphism along a foliation and applications},
  Algebraic geometry, Bowdoin, 1985 (Brunswick, Maine, 1985), Proc. Sympos.
  Pure Math., vol.~46, Amer. Math. Soc., Providence, RI, 1987, pp.~245--268.
  \MR{MR927960 (89e:14011)}

\bibitem[Ray72]{Ray72}
Mich\'ele Raynaud, \emph{Flat modules in algebraic geometry}, Comp. Math. 24, 1972, pp.~11--31.

\bibitem[Yau77]{Yau77}
Shing-Tung Yau, \emph{Calabi's conjecture and some new results in algebraic geometry}, Proc. Natl. Acad. Sci., Vol. 74, No. 5, 1977, pp.~ 1798--1799.


\end{thebibliography}
\end{document}